\documentclass[a4paper,12pt]{amsart}
\usepackage{amssymb,amscd}
\usepackage[cp1251]{inputenc}
\usepackage[T2A]{fontenc}
\usepackage{amsmath}
\usepackage{amsthm}
\usepackage{mathtext}
\usepackage{euscript}

\usepackage[matrix,arrow,curve]{xy}

\usepackage{graphicx}
\graphicspath{}
\DeclareGraphicsExtensions{.pdf,.png,.jpg,.jpeg}

\linethickness{0.5pt}

\usepackage[english]{babel}

\newtheorem{theo}{Theorem}[section]
\newtheorem{prop}[theo]{Proposition}
\newtheorem{lemm}[theo]{Lemma}
\newtheorem{coro}[theo]{Corollary}

\theoremstyle{definition}
\newtheorem{defi}{Definition}
\newtheorem{exam}[theo]{Example}

\newtheorem{constr}[theo]{Construction}

\theoremstyle{remark}
\newtheorem*{rema}{Remark}

\numberwithin{equation}{section}

\newcommand{\mb}[1]{{\textbf {\textit#1}}}

\newcommand{\field}[1]{\mathbb{#1}}
\def\C{\field{C}}

\newcommand{\R}{\field{R}}

\DeclareMathOperator{\Tor}{Tor}
\DeclareMathOperator{\ko}{k}

\DeclareMathOperator{\pr}{pr}
\DeclareMathOperator{\cc}{cc}
\DeclareMathOperator{\fc}{fc}

\DeclareMathOperator{\Link}{Link}
\DeclareMathOperator{\im}{im}

\def\ge{\geqslant}

\newcommand{\zp}{\mathcal Z_P}

\newcommand{\zk}{\mathcal Z_K}

\begin{document}

\title[EMBEDDINGS OF MOMENT-ANGLE MANIFOLDS]{EMBEDDINGS OF MOMENT-ANGLE MANIFOLDS AND SEQUENCES OF MASSEY PRODUCTS}

\author{Victor Buchstaber}
\address{Department of Mathematics and Mechanics, Moscow
State University, 1 Leninskie Gory, Moscow, Russia}
\address{Steklov Mathematical Institute of the Russian Academy of Sciences, 8 Gubkina street 8, Moscow, Russia}
\email{buchstab@mi.ras.ru}

\author{Ivan Limonchenko}
\address{School of Mathematical Sciences, Fudan University, 220 Handan Road, Shanghai, P.R. China}
\email{ilimonchenko@fudan.edu.cn}

\thanks{The first author was supported by the Russian Foundation for Basic Research, grant no.~16-51-55017. The second author was supported by the General Financial Grant from the China Postdoctoral Science Foundation, grant no.~2016M601486.}



\maketitle
\begin{abstract}
We show that for any face $F$ of a simple polytope $P$ the canonical equivariant homeomorphisms $h_P:\,\zp\to\mathcal Z_{K_P}$ and $h_F:\,\mathcal Z_F\to\mathcal Z_{K_F}$ are linked in a pentagonal commutative diagram with the maps of moment-angle manifolds and moment-angle-complexes, induced by a face embedding $i_{F,P}:\, F\to P$ and a simplicial embedding $\Phi_{F,P}:\, K_F\to K_{F,P}\to K_P$, where $K_{F,P}$ is the full subcomplex of $K_P$ on the same vertex set as $\Phi_{F,P}(K_F)$.
We introduce the explicit constructions of the maps $i_{F,P}$, $\Phi_{F,P}$ and show that a polytope $P$ is flag if and only if the induced embedding $\hat{i}_{F,P}:\,\mathcal Z_{F}\to\mathcal Z_P$ of moment-angle manifolds has a retraction and thus induces a split ring epimorphism in cohomology for any face $F\subset P$. As the applications of these results we obtain the sequences $\{P^n\}$ of flag simple polytopes such that there exists a nontrivial $k$-fold Massey product in $H^*(\mathcal Z_{P^n})$ with $k\to\infty$ as $n\to\infty$ and, moreover, the existence of a nontrivial $k$-fold Massey product in $H^*(\mathcal Z_{P^n})$ implies existence of a nontrivial $k$-fold Massey product in $H^*(\mathcal Z_{P^l})$ for any $l>n$.
\end{abstract}

\section{Introduction}

Toric Topology opened new directions in the framework of equivariant topology and found remarkable applications due to the possibility of constructing explicitly smooth manifolds, CW complexes, and their mappings that provide us with realizations of fundamental results of algebraic topology. In particular, the notion of a Massey product, well-known and widely used in algeraic topology, acquired a deep new meaning in toric topology, thanks to the new methods that have led to explicit constructions of manifolds with nontrivial higher Massey products in cohomology.

Consider an embedding $i_{F,P}:\,F\subset P$ of an $r$-dimensional face $F$ into an $n$-dimensional simple polytope $P$. A moment-angle manifold $\mathcal Z_{F}$ of the face $F$ is a submanifold of $\mathcal Z_P$. The explicit construction of an induced embedding of moment-angle manifolds $\hat{i}_{F,P}:\,\mathcal Z_{F}\to\mathcal Z_P$ plays a key role in this paper.

Suppose $S$ is a subset of vertices of a simplicial complex $K$. Denote by $K_S$ the full subcomplex of $K$ on the vertex set $S$.
It is well known that for the canonical embedding $j_{S}:\,K_S\to K$ there exists an induced embedding of moment-angle-complexes $\hat{j}_{S}:\,\mathcal Z_{K_S}\to\zk$ and, furthermore, $K_S$ is a retract of $K$; the explicit construction of this retraction is described in Construction~\ref{fullsubcompretract}.

On the other hand, due to~\cite{bu-pa00-2}, for any simple polytope $P$ there exists a canonical equivariant homeomorphism $h_{P}:\,\zp\to\mathcal Z_{K_P}$ (see Definition~\ref{mac}, Definition~\ref{mamfdDJ}, and Definition~\ref{mamfdBP}), where $K_P$ is the boundary of the dual simplicial polytope $P^*$.

In this work, to any face embedding $i_{F,P}:\,F\subset P$ we associate a canonical simplicial embedding $\Phi_{F,P}:\,K_{F}\to K_{P}$. Consider the full subcomplex $K_{F,P}$ of $K_P$ on the same vertex set as $\Phi_{F,P}(K_{F})$. We prove that the canonical equivariant homeomorphisms $h_{F}, h_{P}$ and the canonical retraction $r:\,\zk\to\mathcal Z_{K_S}$ are connected by the induced embeddings $\hat{i}_{F,P}$, $\hat{\Phi}_{F,P}$ of moment-angle manifolds and moment-angle-complexes, respectively, in a commutative pentagonal diagram with the vertices $\mathcal Z_{K_F}$, $\mathcal Z_{K_P}$, $\mathcal Z_F$, $\mathcal Z_{P}$, and $\mathcal Z_{K_{F,P}}$, see Proposition~\ref{GenCase}.

Moreover, we prove that a polytope $P$ is flag if and only if the embedding $\Phi_{F,P}$ gives an isomorphism $K_{F}\cong K_{F,P}$ for any face $F\subset P$, see Theorem~\ref{FlagCriterion}. As a corollary we obtain that in the latter case the induced embedding  $\hat{i}_{F,P}$ of moment-angle manifolds has a retraction and thus induces a split ring epimorphism in cohomology. 
As an application of this result we obtain sequences $\{P^n\}$ of flag simple polytopes such that there exists a nontrivial $k$-fold Massey product in $H^*(\mathcal Z_{P^n})$ with $k\to\infty$ as $n\to\infty$. Moreover, the existence of a nontrivial $k$-fold Massey product in $H^*(\mathcal Z_{P^n})$ implies existence of a nontrivial $k$-fold Massey product in $H^*(\mathcal Z_{P^l})$ for any $l>n$.

We started with a sequence of polytopes $\mathcal Q=\{Q^n|\,n\geq 0\}$ (see Definition~\ref{2truncMassey}), introduced by the second author~\cite{L1, L2}, for which it was proved in~\cite{L2} that there exists a nontrivial $n$-fold Massey product $\langle\alpha_{1},\ldots,\alpha_{n}\rangle$ with $\dim\alpha_{i}=3,1\leq i\leq n$ in $H^*(\mathcal Z_{Q^n})$ for any $n\geq 2$. 

In this work we prove that $Q^n$ is a facet of $Q^{n+1}$ for all $n\geq 0$, see Theorem~\ref{Qdfpnm}, and obtain as a corollary that there exists a nontrivial $k$-fold Massey product in $H^*(\mathcal Z_{Q^n})$ for any $2\leq k\leq n$, see Corollary~\ref{AllProducts}, cf.~\cite[Theorem 4.1]{L3}. Using Construction~\ref{familyFlag} we introduce a wide family of sequences of flag polytopes, to which our main results can be applied, see Proposition~\ref{MasseySeqInfinity} and Theorem~\ref{mainMasseysequence}.

{\it{Acknowledgements.}} We are grateful to Taras Panov for stimulating discussions on the results of this work. We also thank Jelena Grbi\'c for asking a question of how to find sequences of polytopes $\mathcal P=\{P^n\}$, different from $\mathcal Q$, for which there exists a nontrivial $k$-fold Massey product in $H^*(\mathcal Z_{P^n})$ with $k\to\infty$ as $n\to\infty$.  

\section{Basic constructions, motivation, and main results}

Here we recall only the notions that are the key ones for our constructions and results. For the definitions, constructions, and results on other notions we are using in this work we refer the reader to the monograph~\cite{TT}. Alongside with our main results, in this section we give extended descriptions of certain constructions from~\cite{TT} motivated by their applications in our work.

\begin{defi}
An \emph{abstract simplicial complex} $K$ on the vertex set $[m]=\{1,2,\ldots,m\}$ is a set of subsets of $[m]$ called its \emph{simplices} such that if $\tau\subset\sigma$ and $\sigma\in K$, then $\tau\in K$.\\
The dimension of $K$ is equal to the maximal value of $|\sigma|-1$. We assume in what follows, unless otherwise is stated explicitly, that there are no \emph{ghost vertices} in $K$, that is $\{i\}\in K$ for all $i\in [m]$.
We call a {\emph{full subcomplex}} $K_J$ on the vertex set $J\subseteq [m]$ the set of all elements in $K$ having vertices in $J$, that is $K_{J}=2^{J}\cap K$.
\end{defi}

\begin{constr}(moment-angle-complex)\label{mac}
Suppose $K$ is an abstract simplicial complex on $[m]$. Then define its \emph{moment-angle-complex} to be
$$
\zk=\cup_{\sigma\in K} (D^2,S^1)^{\sigma},
$$
where $(D^2,S^1)^{\sigma}=\prod\limits_{i=1}^{m}\,Y_{i}$ such that
$Y_{i}=D^2$, if $i\in\sigma$, and $Y_{i}=S^1$, otherwise.

Buchstaber and Panov, see~\cite[Chapter 4]{TT}, proved that $\zk$ is a CW complex for any simplicial complex $K$ and, moreover, gave an alternative description of $\zk$ from the following commutative diagram:
$$\begin{CD}
  \zk @>>>(\mathbb{D}^2)^m\\
  @VVrV\hspace{-0.2em} @VV\rho V @.\\
  \cc(K) @>i_c>> I^m
\end{CD}\eqno 
$$
where $i_{c}:\,\cc(K)\hookrightarrow I^{m}=(I^1,I^1)^{[m]}$ is an embedding of a cubical subcomplex $\cc(K)=(I^1,1)^K$ in $I^{m}=[0,1]^m$ (it is PL homeomorphic to a cone over a barycentric subdivision of $K$), induced by the inclusion of pairs: $(I^1,1)\subset (I^1,I^1)$, and the maps $r$ and $\rho$ are projections onto the orbit spaces of $\mathbb{T}^m$-action induced by the coordinatewise action of $\mathbb{T}^m$ on the unitary complex polydisk $(\mathbb{D}^2)^m$ in $\C^m$.
\end{constr}

In our paper we use the following definition of a simple convex polytope.

\begin{defi}\label{Simplepolytopes}
A \emph{simple convex $n$-dimensional polytope} $P$ in
the Euclidean space $\R^n$ with scalar product
$\langle\;,\:\rangle$ can be defined as a bounded
intersection of $m$ closed halfspaces:
 $$
  P=\bigl\{\mb x\in\R^n\colon\langle\mb a_i,\mb
  x\rangle+b_i\ge0\quad\text{for }
  i=1,\ldots,m\bigr\},\eqno (1)
$$
where $\mb a_i\in\R^n$, $b_i\in\R$. We assume that \emph{facets} of $P$
$$
  F_i=\bigl\{\mb x\in P\colon\langle\mb a_i,\mb
  x\rangle+b_i=0\bigr\},\quad\text{for } i=1,\ldots,m.
$$
are in general position, that is, exactly $n$ of them meet at a single
point (such a point is called a {\emph{vertex}} of $P$).
\end{defi}

In what follows we also assume that there are no redundant inequalities
in $(1)$, that is, no inequality can be removed
from $(1)$ without changing~$P$. The latter condition is equivalent to saying that $K_P$ has no ghost vertices. We also fix the following notation: we denote by $m(P^n)$, or simply, by $m(n)$ the number of facets of a simple polytope $P^n$. Then $K_P$ is an $(n-1)$-dimensional triangulated sphere with $m(n)$ vertices.

\begin{defi}
A minimal nonface of $K$ is a set $J=\{j_1,\ldots,j_k\}\subset [m]$, $|J|=k\geq 2$ such that $K_{J}=\partial\Delta^{k-1}$. We denote the set of all minimal nonfaces of $K$ by $MF(K)$.

A simplicial complex $K$ is called \emph{flag} if $|J|=2$ for any $J\in MF(K)$. A simple polytope $P$ is called \emph{flag} if its nerve complex $K_P$ is flag. Equivalently, if a set of facets of $P$ has an empty intersection: $F_{j_1}\cap\ldots\cap F_{j_k}=\varnothing$, then some pair among these facets also has an empty intersection: $J_{j_{s}}\cap F_{j_t}=\varnothing$ for some $s,t\in [k]$.
\end{defi}

The next construction appeared firstly in the work of Davis and Januszkiewicz~\cite{DJ}.

\begin{constr}(moment-angle manifold I)\label{mamfdDJ}
Suppose $P^n$ is a simple convex polytope with the set of facets $\mathcal F(P^n)=\{F_{1},\ldots,F_{m}\}$.
Denote by $T^{F_{i}}$ a 1-dimensional coordinate subgroup in $T^{F}\cong T^{m}$ for each $1\leq i\leq m$ and $T^{G}=\prod\,T^{F_i}\subset T^{F}$ for a face $G=\cap\,F_{i}$ of a polytope $P^n$. Then the \emph{moment-angle manifold} over~$P$ is defined as a quotient space
$$
\zp=T^{F}\times P^{n}/\sim,
$$
where $(t_{1},p)\sim (t_{2},q)$ if and only if $p=q\in P$ and $t_{1}t_{2}^{-1}\in T^{G(p)}$, $G(p)$ is a minimal face of $P$ which contains $p=q$.
\end{constr}

\begin{rema}
It can be deduced from Construction~\ref{mamfdDJ} that if $P_{1}$ and $P_{2}$ are \emph{combinatorially equivalent}, that is, their face lattices are isomorphic (or, equivalently, $K_{P_1}$ and $K_{P_2}$ are simplicially isomorphic), then $\mathcal Z_{P_{1}}$ and $\mathcal Z_{P_{2}}$ are homeomorphic. The opposite statement is {\emph{not}} true.
\end{rema}

Due to the above remark, in what follows we will often not distinguish between combinatorially equivalent simple polytopes.

Let $A_P$ be the $m\times n$ matrix of row vectors $\mb a_i\in\mathbb{R}^n$, and
let $\mb b_P\in\mathbb{R}^m$ be the column vector of scalars $b_i\in\R$. Then we
can rewrite $(1)$ 
as
\[
  P=\bigl\{\mb x\in\R^n\colon A_P\mb x+\mb b_P\ge\mathbf 0\},
\]
and consider the affine map
\[
  i_P\colon \R^n\to\R^m,\quad i_P(\mb x)=A_P\mb x+\mb b_P.
\]
It embeds $P$ into
\[
  \R^m_\ge=\{\mb y\in\R^m\colon y_i\ge0\quad\text{for }
  i=1,\ldots,m\}.
\]

In the series of works by Buchstaber and Panov moment-angle manifolds were intensively studied by means of algebraic topology, combinatorial commutative algebra, and polytope theory, which started a new area of geometry and topology, toric topology, see~\cite{TT}. They gave the following definition and proved it to be equivalent to the one above.

\begin{constr}(moment-angle manifold II)\label{mamfdBP}
Define a \emph{moment-angle manifold} $\mathcal Z_P$ of a polytope $P$
as a pullback from the commutative diagram
$$\begin{CD}
  \mathcal Z_P @>i_Z>>\C^m\\
  @VVV\hspace{-0.2em} @VV\mu V @.\\
  P @>i_P>> \R^m_\ge
\end{CD}\eqno 
$$
where $\mu(z_1,\ldots,z_m)=(|z_1|^2,\ldots,|z_m|^2)$. The projection $\zp\rightarrow P$ in the above diagram is the quotient map of the canonical action of the compact torus $\mathbb{T}^m$ on $\zp$ induced by the standard action of $\mathbb{T}^m$
\[
  \mathbb T^m=\{\mb z\in\C^m\colon|z_i|=1\quad\text{for }i=1,\ldots,m\}
\]
on~$\C^m$. Therefore, $\mathbb T^m$ acts on $\zp$ with an orbit space $P$, and $i_Z$ is a $\mathbb T^m$-equivariant embedding.
\end{constr}

It follows immediately from the Construction~\ref{mamfdBP}, see~\cite[\S3]{BR}, that $\zp$ is a total intersection of Hermitian quadrics in $\C^m$. Thus, $\zp$ obtains a canonical equivariant smooth structure. For any simple $n$-dimensional polytope $P$ with $m$ facets its moment-angle manifold $\zp$ is a 2-connected, closed, $(m+n)$-dimensional manifold.

\begin{rema}
The general problem on how many $K$-invariant smooth structures may exist on $\zp$, where $K\cong T^{m-n}$ is a maximal subgroup of $\mathbb{T}^m$ that acts freely on $\zp$, remains open. In the case when $P=\Delta^3$ and $K\cong S^1$ (that is, $\zp=S^7$) it was solved in the work of Bogomolov~\cite{Bg}.
\end{rema}

Now we focus on a description of the cohomology algebra of $\zp$ and higher Massey products in it. To proceed, we need firstly to recall some notions from combinatorial commutative algebra.

Let $\ko$ be a commutative ring with a unit. Throughout the paper, unless otherwise stated explicitly, we denote by $K$ an $(n-1)$-dimensional simplicial complex on the vertex set $[m]=\{1,2,\ldots,m\}$ and by $P$ a simple convex $n$-dimensional polytope with $m$ facets: $\mathcal F(P)=\{F_{1},\ldots,F_{m}\}$. Let $\ko[m]=\ko[v_1,\ldots,v_m]$ be the graded polynomial algebra on $m$ variables, $\deg v_{i}=2$.

\begin{defi}\label{Facerings}
A \emph{face ring} (or a \emph{Stanley-Reisner ring}) of $K$ is the quotient ring
$$
   \ko[K]:=\ko[v_{1},\ldots,v_{m}]/I_K
$$
where $I_K$ is the ideal generated by those square free
monomials $v_{i_{1}}\cdots{v_{i_{k}}}$ such that $\{i_{1},\ldots,i_{k}\}\notin K$.\\
We call a \emph{face ring of a polytope} $P$ the Stanley-Reisner ring of its {\emph{nerve complex}}: $\ko[P]=\ko[K_P]$, where $K_P=\partial P^*$.
\end{defi}

Note that $\ko[P]$ is a module over $\ko[v_{1},\ldots,{v_{m}}]$ via
the quotient projection. The following result relates the cohomology algebra of $\zk$ to combinatorics of $K$.

\begin{theo}[{\cite[Theorem 4.5.4]{TT} or \cite[Theorem 4.7]{P}}]\label{BPtheo}
The following statements hold.
\begin{itemize}
\item[(I)] The isomorphisms of algebras hold:
$$
\begin{aligned}
  H^*(\zk;\ko)&\cong\Tor_{\ko[v_1,\ldots,v_m]}^{*,*}(\ko[K],\ko)\\
  &\cong H^{*,*}\bigl[\Lambda[u_1,\ldots,u_m]\otimes \ko[K],d\bigr]\\
  &\cong \bigoplus\limits_{J\subset [m]}\widetilde{H}^*(K_{J};\ko),
\end{aligned}
$$
$\mathop{\mathrm{bideg}} u_i=(-1,2),\;\mathop{\mathrm{bideg}} v_i=(0,2);\quad
  du_i=v_i,\;dv_i=0$.
The last isomorphism is the sum of isomorphisms of $\ko$-modules:
$$
H^p(\zp;\ko)\cong\sum\limits_{J\subset [m]}\widetilde{H}^{p-|J|-1}(P_{J};\ko);
$$

\item[(II)] Cup product in $H^*(\zk;\ko)$ is described as follows.
Suppose we have two cohomology classes $\alpha=[a]\in\tilde{H}^{p}(K_{I_1};\ko)$ and $\beta=[b]\in\tilde{H}^{q}(K_{I_2};\ko)$ on full subcomplexes $K_{I_1}$ and $K_{I_2}$. Then one can define a natural inclusion of sets $i:\,K_{I_{1}\sqcup I_{2}}\rightarrow K_{I_1}*K_{I_2}$ and the canonical isomorphism of cochain modules:
$$
s:\,\tilde{C}^{p}(K_{I_1})\otimes\tilde{C}^{q}(K_{I_2})\rightarrow\tilde{C}^{p+q+1}(K_{I_{1}}*K_{I_2}).
$$
Then the product of $\alpha$ and $\beta$ is given by:
$$
\alpha\cdot\beta=\begin{cases}
0,&\text{if $I_{1}\cap I_{2}\neq\varnothing$;}\\
i^{*}[s(a\otimes b)]\in\tilde{H}^{p+q+1}(K_{I_{1}\sqcup I_{2}};\ko),&\text{if $I_{1}\cap I_{2}=\varnothing$.}
\end{cases}
$$
\end{itemize}
\end{theo}

\begin{rema}
Hochster~\cite{Hoch} proved an isomorphism of $\ko$-modules
$$
\Tor_{\ko[v_1,\ldots,v_m]}^{*,*}(\ko[K],\ko)\cong\bigoplus\limits_{J\subset [m]}\widetilde{H}^*(K_{J};\ko)
$$
for any simplicial complex $K$. Since one has a homeomorphism $\zp\cong\mathcal Z_{K_P}$, this gives a description of the additive structure of $H^*(\zp;\ko)$ as well.
\end{rema}

If we define a finitely generated differential graded algebra $R(K)=\Lambda[u_{1},\ldots,u_{m}]\otimes\ko[K]/(v_{i}^{2}=u_{i}v_{i}=0,1\leq i\leq m)$ with the same $d$ as in the theorem above, then one has the following result.

\begin{prop}\label{BPmultigrad}
A graded algebra isomorphism holds:
$$
H^{*,*}(\zk;\ko)\cong H^{*,*}[R(K),d]\cong\Tor^{*,*}_{\ko[v_{1},\ldots,v_{m}]}(\ko[K],\ko).
$$
These algebras admit $\mathbb{N}\oplus\mathbb{Z}^m$-multigrading and we have
$$
\Tor^{-i,2{\bf{a}}}_{\ko[v_{1},\ldots,v_{m}]}(\ko[K],\ko)\cong H^{-i,2{\bf{a}}}[R(K),d],
$$
where $i\in\mathbb{N}, J\in\mathbb{Z}^m$, and $\Tor^{-i,2J}_{\ko[v_{1},\ldots,v_{m}]}(\ko[K],\ko)\cong\widetilde{H}^{|J|-i-1}(K_{J};\ko)$
for $J\subseteq [m]$. The multigraded component $\Tor^{-i,2{\bf{a}}}_{\ko[v_{1},\ldots,v_{m}]}(\ko[K],\ko)=0$, if ${\bf{a}}$ is not a $(0,1)$-vector of length $m$.
\end{prop}

\begin{constr}\label{fullsubcompretract}
Consider a subset of vertices $S\subset [m]$ of a simplicial complex $K$. Then there is a natural simplicial embedding $j_{S}:\,K_{S}\to K$. We are going to construct a retraction to the induced embedding of moment-angle-complexes $\hat{j}_{S}:\mathcal Z_{K_S}\to\mathcal\zk$. To do this, consider a projection
$$
p_{\sigma}:\,\prod\limits_{i\in\sigma}\,D^{2}\times\prod\limits_{i\notin\sigma}\,S^{1}\rightarrow\prod\limits_{i\in\sigma\cap S}D^{2}\times\prod\limits_{S\backslash\sigma}S^{1}
$$
for each $\sigma\in K$. Observe that the image of $p_{\sigma}$ is in $\mathcal Z_{K_S}$ for all $\sigma\in K$, since $K_{S}=\{\sigma\cap S|\,\sigma\in K\}$. It is easy to see that $r_{S}:\,\cup_{\sigma\in K}\,p_{\sigma}:\,\zk\to\mathcal Z_{K_S}$ is a retraction.
\end{constr}

\begin{coro}\label{splitepi}
Suppose $j_{S}:\,K_{S}\hookrightarrow K$ is an embedding of a full subcomplex of $K$ on the vertex set $S\subseteq [m]$. Then the embedding of moment-angle-complexes $\hat{j}_{S}:\,\mathcal Z_{K_S}\to\zk$ has a retraction map and the induced ring homomorphism in cohomology $j^*_{S}:\,H^*(\zk)\rightarrow H^*(\mathcal Z_{K_S})$ is a split ring epimorphism.
\end{coro}
\begin{proof}
The statement about the embedding of moment-angle-complexes follows directly from Definition~\ref{mac} and Construction~\ref{fullsubcompretract} (cf.~\cite[Exercise 4.2.13]{TT}). The rest of the statement now follows from the fact that a homomorphism in cohomology, induced by a retraction, is a split ring epimorphism: $(ri)^{*}=i^{*}r^{*}=1_{A}^{*}$ for a retract $A=\mathcal Z_{K_S}\subset\zk$.
\end{proof}

Now we will discuss in more details different ways to construct an equivariant embedding of a moment-angle manifold $\mathcal Z_{F^r}$ into a moment-angle manifold $\mathcal Z_{P^n}$ induced by a face embedding $F^{r}\to P^{n}$. Although, $\mathcal Z_{F^r}$ is always a submanifold of $\mathcal Z_{P^n}$, no retraction $\mathcal Z_{P^n}\to\mathcal Z_{F^r}$ exists, in general, see Example~\ref{prism} below.

\begin{constr}(mappings of moment-angle manifolds I)\label{mapmfds}
Suppose $F^{r}=F_{i_{1}}\cap\ldots\cap F_{i_{n-r}}$ is an $r$-dimensional face of $P^n$ and $i^{n}_{r}:\,F^{r}\hookrightarrow\partial P^{n}\subset P^n$ is its embedding into $P^n$. Then $F^r$ is a simple polytope itself having $m(r)$ facets $G_{i},1\leq i\leq m(r)$, that is for each facet $G_{\alpha}$ of $F^r$ there is a unique facet $F_{j}$ of $P^n$ such that
$$
G_{\alpha}=(F_{i_{1}}\cap\ldots\cap F_{i_{n-r}})\cap F_{j}
$$
Thus, a map $\phi^{n}_{r}:\,[m(r)]\rightarrow [m(n)]$ is determined such that $\phi^{n}_{r}(\alpha)=j$.

Now using Construction~\ref{mamfdDJ}, we are going to construct a map $\hat{\phi}^{n}_{r}:\,\mathcal Z_{F^r}\rightarrow\mathcal Z_{P^n}$ induced by $\phi^{n}_{r}$ and $i^{n}_{r}$.

First consider the following map
$$
\tilde{\phi}^{n}_{r}:\,T^{\mathcal F(F^r)}\rightarrow T^{\mathcal F(P^n)},
$$
where $\mathcal F(F^r)=\{G_{1},\ldots,G_{m(r)}\}$ and $\mathcal F(P^n)=\{F_{1},\ldots,F_{m(n)}\}$ denote the sets of facets of $F^r$ and $P^n$ respectively, and
$$
\tilde{\phi}^{n}_{r}(t_{1},\ldots,t_{m(r)})=(\tau_{1},\ldots,\tau_{m(n)}),
$$
for
$$
\tau_{i}=\begin{cases}
t_{(\phi^{n}_{r})^{-1}(i)},&\text{if $i\in\im\phi_{r}^{n}$;}\\
1,&\text{otherwise.}
\end{cases}
$$

It is easy to see that $\tilde{\phi}^{n}_{r}:\,T^{m(r)}\rightarrow T^{m(n)}$ is a group homomorphism.

Finally, we are able to define a map
$$
\hat{\phi}^{n}_{r}:\,\mathcal Z_{F^r}=(T^{\mathcal F(F^r)}\times F^{r})/\sim\rightarrow\mathcal Z_{P^n}=(T^{\mathcal F(P^n)}\times P^{n})/\sim
$$
by formula:
$$
\hat{\phi}^{n}_{r}([t,p])=[\tilde{\phi}^{n}_{r}(t),i_{r}^{n}(p)]
$$

Let us prove correctness of the definition above. Due to Construction~\ref{mamfdDJ} it sufficies to prove that if $T_{1}^{-1}T_{2}\in T^{G_{F^r}(p=q)}$ then $\tilde{\phi}^{n}_{r}(T_{1}^{-1})\tilde{\phi}^{n}_{r}(T_{2})\in T^{G_{P^n}(p=q)}$.
As $\tilde{\phi}^{n}_{r}$ is a group homomorphism, we need to prove that
$$
\tilde{\phi}^{n}_{r}(T_{1}^{-1}T_{2})\in T^{G_{P^n}(p=q)},
$$
whenever
$$
T_{1}^{-1}T_{2}\in T^{G_{F^r}(p=q)}.
$$

Let $T_{1}^{-1}T_{2}=(t_{1},\ldots,t_{m(r)})$ and $\tilde{\phi}^{n}_{r}(T_{1}^{-1}T_{2})=(\tau_{1},\ldots,\tau_{m(n)})$.
Note that, by Construction~\ref{mamfdDJ}, $G:=G_{F^r}(p=q)=G_{P^n}(p=q)$ is the unique face in $F^r\subset P^n$ for which the point $p=q$ belongs to its interior.
Suppose $G=G_{\alpha_{1}}\cap\ldots\cap G_{\alpha_{k}}=F^{r}\cap F_{j_{1}}\cap\ldots\cap F_{j_k}$. Then by definition of $\phi^{n}_{r}$ one has:
$\phi^{n}_{r}(\alpha_{p})=j_p$ for all $1\leq p\leq k$.

Now it suffices to show that if $\tau_{i}\neq 1$, then $i\in\{j_{1},\ldots,j_{k}\}$.
As $\tau_{i}\neq 1$, by definition of $\tilde{\phi}^{n}_{r}$ one has: $\tau_{i}=t_{(\phi^{n}_{r})^{-1}(i)}\neq 1$ and $i\in\im\phi^{n}_{r}$.
Since $t_{(\phi^{n}_{r})^{-1}(i)}\neq 1$ and $(t_{1},\ldots,t_{m(r)})\in T^{G}=T^{G_{\alpha_{1}}}\times\ldots\times T^{G_{\alpha_{k}}}$, we must have $t_{(\phi^{n}_{r})^{-1}(i)}\in T^{G_{\alpha_{s}}}$ for some $s\in [k]$. Again by definition of $\tilde{\phi}^{n}_{r}$ this means that $(\phi^{n}_{r})^{-1}(i)=\alpha_{s}$, or $\phi^{n}_{r}(\alpha_{s})=i$. As we previously had $\phi^{n}_{r}(\alpha_{s})=j_{s}$ for $G$ above, it implies that $i=\phi^{n}_{r}(\alpha_{s})=j_{s}\in\{j_{1},\ldots,j_k\}$, and correctness of the definition of $\hat{\phi}^{n}_{r}$ is proved.

To show that $\hat{\phi}^{n}_{r}$ is continuous, we observe that the following commutative diagram takes place, by definition of $\hat{\phi}^{n}_{r}$:
$$\begin{CD}
  T^{m(r)}\times F^{r} @>\tilde{\phi}^{n}_{r}\times i_{r}^{n}>> T^{m(n)}\times P^{n}\\
  @VV\pr_{r} V\hspace{-0.2em} @VV\pr_{n} V @.\\
  \mathcal Z_{F^r} @>\hat{\phi}^{n}_{r}>>\mathcal Z_{P^n}
\end{CD}\eqno 
$$
Since $\mathcal Z_{F^r}$ and $\mathcal Z_{P^n}$ both have quotient topologies determined by the canonical projections $pr_{r}$ and $pr_{n}$, respectively, given in Construction~\ref{mamfdDJ}, one concludes that $\hat{\phi}^{n}_{r}$ is continuous if and only if $\hat{\phi}^{n}_{r}\pr_{r}=\pr_{n}\;(\tilde{\phi}^{n}_{r}\times i_{r}^{n})$ is continuous. The latter map is a composition of continuous maps, which finishes the proof.
\end{constr}

Now we will construct a simplicial map $\Phi^{n}_{r}:\,K_{F^r}\rightarrow K_{P^n}$ determined by $\phi_{r}^{n}$ such that it induces a continuous map of moment-angle-complexes $\hat{\Phi}^{n}_{r}:\,\mathcal Z_{K_{F^r}}\rightarrow\mathcal Z_{K_{P^n}}$.

\begin{constr}\label{mapmcxs}
First, recall that $\phi^{n}_{r}$ induces an injective map of facets $\overline{\phi}^{n}_{r}:\,\mathcal F(F^r)\rightarrow\mathcal F(P^n)$, where $\overline{\phi}^{n}_{r}(G_{\alpha})=F_{\phi^{n}_{r}(\alpha)}$ for $G_{\alpha}=F^r\cap F_{\phi^{n}_{r}(\alpha)}$.

Then any simplex $\sigma=(\alpha_{1},\ldots,\alpha_k)\in K_{F^r}$ is in 1-1 correspondence with a nonempty intersection of facets of $F^r$:
$$
G_{\alpha_1}\cap\ldots\cap G_{\alpha_k}=F^{r}\cap F_{\phi^{n}_{r}(\alpha_{1})}\cap\ldots\cap F_{\phi^{n}_{r}(\alpha_k)}\neq\varnothing.
$$
It implies that $F_{\phi^{n}_{r}(\alpha_{1})}\cap\ldots\cap F_{\phi^{n}_{r}(\alpha_k)}\neq\varnothing$, thus a nondegenerate injective simplicial map $\Phi^{n}_{r}:\,K_{F^r}\rightarrow K_{P^n}$ can be defined by
$$
\Phi^{n}_{r}(\sigma)=(\phi_{r}^{n}(\alpha_{1}),\ldots,\phi^{n}_{r}(\alpha_k))\in K_{P^n}.
$$
Observe that:
\begin{itemize}
\item[(1)] If $F^{r}=F_{i_{1}}\cap\ldots\cap F_{i_{n-r}}$ in $P$, then there is a simplex $\Delta({F})=\{i_{1},\ldots,i_{n-r}\}\in K_P$;
\item[(2)] By definition of link, $\Phi_{r}^{n}(K_{F^r})=\Link_{K_{P^n}}\Delta(F)$. 
\end{itemize} 

It follows from Construction~\ref{mac} that $\Phi^{n}_{r}$ determines a continuous map of moment-angle-complexes:
$$
\hat{\Phi}^{n}_{r}:\,\mathcal Z_{K_{F^r}}\rightarrow\mathcal Z_{K_{P^n}},
$$
which is induced by homeomorphisms:
$$
(D^2,S^1)^{\sigma}\cong (D^2,S^1)^{\Phi^{n}_{r}(\sigma)}.
$$
\end{constr}

\begin{rema}
Note that both $\hat{\phi}^{n}_{r}$ and $\hat{\Phi}^{n}_{r}$ are weakly equivariant maps with respect to the $\mathbb{T}^{m(r)}$-action, induced by the map $\tilde{\phi}^{n}_{r}:\,\mathbb{T}^{m(r)}\rightarrow\mathbb{T}^{m(n)}$.
\end{rema}

Now, using Definition~\ref{mamfdBP}, we are going to introduce a description of a map of moment-angle manifolds $\mathcal Z_{F^r}\rightarrow\mathcal Z_{P^n}$, induced by a face embedding $i_{r}^{n}:\,F^r\to P^n$, equivalent to that in Construction~\ref{mapmfds}, for which we are able to give explicit formulae in coordinates in the ambient complex Euclidean spaces $\C^{m(r)}$ and $\C^{m(n)}$.

\begin{constr}(mappings of moment-angle manifolds II)\label{mapmfds2}
Suppose $F^{r}$ is an $r$-dimensional face of $P^n$ with a set of facets $\mathcal F(P^n)=\{F_{1},\ldots,F_{m(n)}\}$.

Firstly, let us prove that there exists an induced embedding $\hat{i}_{r}^{n}:\,\mathcal Z_{F^r}\hookrightarrow\mathcal Z_{P^n}$.
Consider the affine embedding $f:\,\R^n \to\R^{m(n)}$ such that its restriction to $P^n$ equals $i_{P^n}$, see Construction~\ref{mamfdDJ}, and the embedding $g:\,\R^r \to\R^n$, whose restriction on $F^r$ is $i_{r}^{n}$.
Now consider the composition map $f\,g:\,\R^r \to \R^n \to R^{m(n)}$ and a section $s_{n}: \R^{m(n)}_\ge \to \C^{m(n)}$.
Recall that the embedding $\tilde{\phi}_{r}^{n}:\,T^{m(r)}\to T^{m(n)}$ introduced in Construction~\ref{mapmfds} gives an action of $T^{m(r)}$ on $\C^{m(n)}$. Finally, we get an action of $T^{m(r)}$ on the image $s_{n}\,f\,g(F^r)$ of the face $F^r$ in $\C^{m(n)}$, where $s_{n}$ is a continuous section of the moment map $\mu_{n}:\,\C^{m(n)}\to\R^m_{\ge}$. Let us denote the corresponding moment-angle manifold (see Construction~\ref{mamfdDJ}) by $W$. Then, by the Construction~\ref{mamfdBP}, $W$ is embedded into $\mathcal Z_{P^n}$ in $\C^{m(n)}$ and the induced map $\hat{i}_{r}^{n}:\,\mathcal Z_{F^r}\to\mathcal Z_{P^n}$ is defined, whose image is $W$.

Now let us give explicit formulae for the embedding of $W$ into $\C^{m(n)}$.

Without loss of generality, we may assume that $F^{r}=F_{r+1}\cap\ldots\cap F_{n}$ and $P^n$ is given in $\R^{n}$ by the following system of linear inequalities:
$$
P^{n}=\{x\in\R^n|\,A_{P^n}x+b_{P^n}\geq 0\}, A_{P^n}^{T}=(E_{n},\tilde{A}^{T}), b_{P^n}=(\underbrace{0,\ldots,0}_{n},\tilde{b})^{T},
$$
where $\tilde{A}$ is an $(m(n)-n)\times n$-matrix and $\tilde{b}=(b^{1},\ldots b^{m(n)-n})$.

Then one has:
$$
i_{r}^{n}:\,F^{r}\rightarrow P^{n}, i_{r}^{n}(x^{1},\ldots,x^{r})=(x^{1},\ldots,x^{r},0_{n-r}),
$$
where $0_{n-r}=(\underbrace{0,\ldots,0}_{n-r})$.
Denote the submatrix consisting of the first $r$ columns and first $m(r)$ rows of $\tilde{A}$ by $\tilde{A}_{I}$ and the submatrix consisting of the first $r$ columns and last $m(n)-m(r)-n$ rows of $\tilde{A}$ by $\tilde{A}_{II}$. Set $\tilde{b}=(b_{I},b_{II})^{T}$, where $b_{I}$ has length $m(r)$ and $b_{II}$ has length $m(n)-m(r)-n$. Observe that the following equality holds:
$$
F^{r}=\{x\in\R^r|\,\tilde{A}_{I}x+b_{I}\geq 0\}.
$$

Note that the following formula takes place:
$$
i_{P^n}\,i_{r}^{n}(x)=(x,0_{n-r},\tilde{A}_{I}x+b_{I},\tilde{A}_{II}x+b_{II}),
$$
where $x=(x^{1},\ldots,x^{r})^{T}$.
On the other hand, one gets by definition:
$$
i_{F^r}(x)=\tilde{A}_{I}x+b_{I}.
$$

Since $F^r$ is a simple polytope, the rank of $\tilde{A}_{I}$ equals $r$.
Thus, there exist a $r\times m(r)$-matrix $C$ and a column vector $D$ of length $r$ such that
$$
y=\tilde{A}_{I}x+b_{I}\,\rightarrow\,x=Cy+D,
$$
where $x\in\R^r, y\in\R^{m(r)}$.

Consider an affine embedding $i_{R}:\,\mathbb{R}^{m(r)}\rightarrow\mathbb{R}^{m(n)}$ given by formula
$$
i_{R}(y)=(Cy+D,0_{n-r},y,(\tilde{A}_{II}C)y+(\tilde{A}_{II}D+b_{II})).
$$

Now the following diagram commutes:
$$\begin{CD}
  F^r @>i_{r}^{n}>> P^n\\
  @VVi_{F^r} V\hspace{-0.2em} @VVi_{P^n} V @.\\
  \R^{m(r)} @>i_{R}>>\R^{m(n)}
\end{CD}\eqno 
$$

Observe that for a certain map $i_{C}:\C^{m(r)}\rightarrow\C^{m(n)}$ the following diagram commutes:
$$\begin{CD}
  \mathcal Z_{F^r} @>i_{\mathcal Z_{F^r}}>> \mathbb{C}^{m(r)} @>i_{C}>> \mathbb{C}^{m(n)}\\
  @VV\pr_{\mathbb{T}^r} V\hspace{-0.2em} @VV\mu_{m(r)} V\hspace{-0.2em} @VV\mu_{m(n)} V @.\\
  F^r @>i_{F^r}>>\mathbb{R}^{m(r)}_{\ge} @>i_{R}>> \mathbb{R}^{m(n)}_{\ge}.
\end{CD}\eqno 
$$

By Construction~\ref{mamfdBP} $\mathcal Z_{P^n}$ is a pullback in the following diagram:
$$\begin{CD}
  \mathcal Z_{P^n} @>i_{\mathcal Z_{P^n}}>>\C^{m(n)}\\
  @VVV\hspace{-0.2em} @VV\mu_{m(n)} V @.\\
  P^{n} @>i_{P^n}>> \R^{m(n)}_{\ge}
\end{CD}\eqno 
$$

Therefore, by the universal property of pullback, there exists a unique map $\hat{i}_{r}^{n}:\,\mathcal Z_{F^r}\rightarrow\mathcal Z_{P^n}$ such that the following diagram commutes:
$$\begin{CD}
  \mathcal Z_{F^r} @>\hat{i}_{r}^{n}>> \mathcal Z_{P^n} @>i_{\mathcal Z_{P^n}}>> \mathbb{C}^{m(n)}\\
  @VV\pr_{\mathbb{T}^r} V\hspace{-0.2em} @VV\pr_{\mathbb{T}^
{n}} V\hspace{-0.2em} @VV\mu_{m(n)} V @.\\
  F^r @>i_{r}^{n}>> P^{n} @>i_{P^n}>> \mathbb{R}^{m(n)}_{\ge}.
\end{CD}\eqno 
$$

In particular, from the diagram above one has: $i_{\mathcal Z_{P^n}}\,\hat{i}_{r}^{n}:\,\mathcal Z_{F^r}\rightarrow\mathcal Z_{P^n}\hookrightarrow\C^{m(n)}$ coincides with the composition of embeddings
$i_{C}\,i_{\mathcal Z_{F^r}}:\,\mathcal Z_{F^r}\hookrightarrow\C^{m(r)}\hookrightarrow\C^{m(n)}$. Therefore, $\hat{i}_{r}^{n}:\,\mathcal Z_{F^r}\rightarrow\mathcal Z_{P^n}$ is an embedding.

It follows that $\hat{i}_{r}^{n}(z)$ for $z=(z_{1},\ldots,z_{m(r)})\in\mathcal Z_{F^r}$ coincides with $i_{C}(z)$ (we regard $\mathcal Z_{F^r}$ as an intersection of Hermitian quadrics in $\C^{m(r)}$), and, moreover, one has: $W=i_{\mathcal Z_{P^n}}\hat{i}_{r}^{n}(\mathcal Z_{F^r})$. Therefore, we can give $W\subset\C^{m(n)}$ by the following formulae:
$$
(|z_{1}|^{2},\ldots,|z_{r}|^{2})=(x_{1},\ldots,x_{r})=:x,\quad (z_{r+1},\ldots,z_{n})=0_{n-r},
$$
$$
(|z_{n+1}|^{2},\ldots,|z_{n+m(r)}|^{2})=\tilde{A}_{I}x+b_{I},\quad (|z_{n+m(r)+1}|^{2},\ldots,|z_{m(n)}|^{2})=\tilde{A}_{II}x+b_{II}.
$$
For the embedding $i_{C}:\C^{m(r)}\to\C^{m(n)}$ one has the following formulae:
$$
\C^{m(r)}\to\C^{m(n)}:\,z\to(\sqrt{C_{1}Z+D_{1}},\ldots,\sqrt{C_{r}Z+D_{r}},0_{n-r},z,
$$
$$
\sqrt{(\tilde{A}_{II}C)_{1}Z+(\tilde{A}_{II}D)_{1}+b_{II,1}},\ldots,
$$
$$
\sqrt{(\tilde{A}_{II}C)_{m(r,n)}Z+(\tilde{A}_{II}D)_{m(r,n)}+b_{II,m(r,n)}}),
$$
where $z=(z_{1},\ldots,z_{m(r)})$, $Z=(|z_{1}|^{2},\ldots,|z_{m(r)}|^{2})$, and $m(r,n)=m(n)-m(r)-n$.
\end{constr}

\begin{exam}
Consider a $5$-gon $P_5^2$, which is embedded into $\mathbb{R}^2$ as shown in the Figure below with facets $F_1,\ldots,F_5$ labeled simply by $\{1,2,3,4,5\}$ respectively.


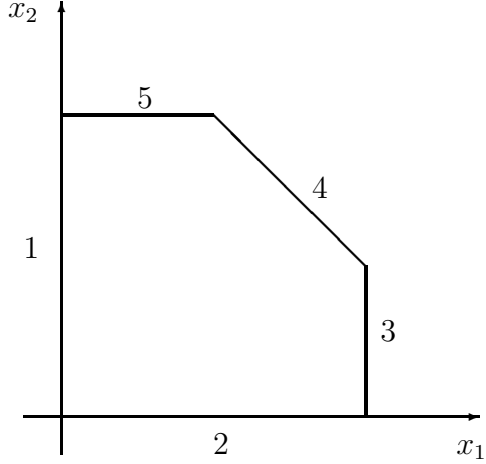
\begin{center}
\begin{picture}(90,60)(0,0)
{\thicklines
  \put(25,40){\line(1,0){20}}
  \put(65,0){\line(0,1){20}}
  \put(65.1,0){\line(0,1){20}}
  \put(65,20){\line(-1,1){20}}
  }
   \put(20,0){\vector(1,0){60}}
   \put(25,-5){\vector(0,1){60}}
  \put(77,-5){$x_1$}
  \put(18,53){$x_2$}
   \put(20,21){$1$}
   \put(45,-5){$2$}
   \put(67,10){$3$}
   \put(58,29){$4$}
   \put(35,41){$5$}
\end{picture}
\end{center}
\vskip 1.0cm

Let us take $P^2=P_5^2 = \{ x\in \mathbb{R}^2\;: Ax+b\geqslant 0 \}$, where
\[ A^\top =
  \begin{pmatrix}
    1 & 0 & -1 & -1 & 0 \\
    0 & 1 & 0 & -1 & -1 \\
  \end{pmatrix},\qquad
 b^\top = (0,0,2,3,2).
\]
and $\top$ means transposition of a matrix. Then its moment-angle manifold is embedded into $\mathbb{C}^5$ with coordinates $(z_1,\ldots,z_5)$ and one obtains $\mathcal{Z}_{P^2}$ as the following intersection of quadrics:\\[7pt]
1. $|z_1|^2+|z_3|^2 = 2$,\\[7pt]
2. $|z_1|^2+|z_2|^2+|z_4|^2 = 3$,\\[7pt]
3. $|z_2|^2+|z_5|^2 = 2$.\\[1pt]

In $\mathbb{R}^2$ with coordinates $(x_1,x_2)$ the facet $F^{r}=F_1\subset P^2$ is given by $x_1=0$. Obviously, facets of $F_1$ are its intersections with $F_2$ and $F_5$.

Let us realize $F_1$ as a line segment $I=\{ x_{2}\in  \mathbb{R}^1\;: 0\leqslant x_{2} \leqslant 2 \}$. Then the face embedding $i^{2}_{1} \colon I\subset P^2$ works as follows: $x_{2} \to (0,x_{2})$. In $\mathbb{C}^2$ with coordinates $(z_2,z_5)$ the moment-angle manifold $\mathcal{Z}_{F_1}$ of the face is a sphere given by the equation $|z_2|^2+|z_5|^2 = 2$.

The embedding of $\mathbb{C}^2$ into $\mathbb{C}^5$, that maps $\mathcal{Z}_{F_1}$ into $\mathcal{Z}_{P^2}$, covers the embedding of the face $i^{2}_{1}:\,I\subset P^2$. Moreover, the corresponding composition of polytope embeddings $j \colon I=F_1 \subset P^2 \to \mathbb{R}_{\geqslant}^5$ is given by: $x_{2} \to (0,x_{2},2,-x_{2}+3,-x_{2}+2)$.

Consider the projection map $\pi \colon \mathbb{C}^5 \to \mathbb{R}_{\geqslant}^5\; : (z_1,\ldots,z_5)
\to (|z_1|^2,\ldots,|z_5|^2 )$ and set $W=\pi^{-1}j(F_1)$. In coordinates one has:
\[
W=\{z\in\mathbb{C}^5\,| z_1=0; |z_2|^2=x_{2}; |z_3|^2 = 2; |z_4|^2 =-x_{2}+3; |z_5|^2 =-x_{2}+2\}.
\]
Here $W$ is the image of $\mathcal Z_{F}$ under the map $i_{C}:\mathbb{C}^{2}\rightarrow\mathbb{C}^{5}$. We get explicit formulae for this map, which is an embedding, although a nonlinear one:
\[
\mathbb{C}^2 \subset \mathbb{C}^5\colon (z_2,z_5) \to (0,z_2,\sqrt{2},\sqrt{3-|z_2|^2},z_5).
\]
Its restriction gives the embedding $\hat{i}^{2}_{1}$ of $\mathcal{Z}_{F_1}=\{ (z_2,z_5)\in \mathbb{C}^2\;: |z_2|^2+|z_5|^2 = 2 \}$ into $\mathcal{Z}_{P^2}$.
\end{exam}

In the notation of the above constructions let us denote by $K^{r-1}=K_{F^r}$, $K^{n-1}=K_{P^n}$, and $K_{n,r}=K^{n-1}_{\phi_{r}^{n}[m(r)]}$. Then due to Construction~\ref{mapmcxs} the nondegenerate injective simplicial map $\Phi_{r}^{n}:\,K^{r-1}\to K^{n-1}$ gives an embedding
$$
\Phi^{n}_{r}(K^{r-1})\subseteq K_{n,r}
$$
and is a composition of simplicial maps
$K^{r-1}\rightarrow K_{n,r}\rightarrow K^{n-1}$, where the former map is induced by $\Phi^{n}_{r}$ and the latter map is a natural embedding of a full subcomplex into its simplicial complex.

Then the following general result holds.

\begin{prop}\label{GenCase}
There is a commutative diagram
$$
\xymatrix{
& \mathcal Z_F\ar[rr]^{\hat{i}^{n}_{r}} \ar[dl]_{h_F} && \mathcal Z_P\ar[dr]^{h_P} \\
\mathcal Z_{K^{r-1}} \ar[drr] &&&& \mathcal Z_{K^{n-1}},\\
&& \mathcal Z_{K_{n,r}} \ar[urr]
}
$$
where $h_{F}$ and $h_{P}$ are equivariant homeomorphisms, and the composition map in the bottom rows equals $\hat{\Phi}^{n}_{r}$.
\end{prop}
\begin{proof}
The above diagram commutes by applying Construction~\ref{mapmfds} to the equivariant homeomorphism $h_{P}:\,\mathcal Z_{P}\cong\mathcal Z_{K_{P}}$ defined in~\cite[Lemma 3.1.6, formula (37)]{bu-pa00-2}, using the cubical subdivision $\mathcal C(P)\subset I^{m(n)}$. Namely, first observe that for any face $F^{r}\subset P^n$ one has: $\mathcal C(F^r)$ is a cubical subcomplex in $\mathcal C(P^n)$, and the following diagram commutes
$$
\begin{CD}
  \mathcal Z_{F^r} @>\hat{i}_{r}^{n}>> \mathcal Z_{P^n} @>h_{P}>> \mathcal Z_{K^{n-1}} @>>> (\mathbb{D}^2)^m\\
  @VVr V\hspace{-0.2em} @VV\rho V\hspace{-0.2em} @VV\rho V\hspace{-0.2em} @VV\rho V @.\\
  F^{r} @>i_{r}^{n}>>P^n @>j_P>> \cc(K^{n-1}) @>>> I^{m(n)},
\end{CD}\eqno 
$$
where $r$ is a projection onto the orbit space of the canonical $\mathbb{T}^{m(r)}$-action, $\rho$ is a projection onto the orbit space of the canonical $\mathbb{T}^{m(n)}$-action, and $j_P$ is an embedding of a cubical subdivision $\mathcal C(P)$ of $P$ into $I^m$ with its image being $\cc(K^{n-1})$, see~\cite{bu-pa00-2}.

Note that the composition map in the bottom row equals $j_F$, since $j_{P}(\mathcal C(F))=\cc(K^{r-1})\subset\cc(K^{n-1})$, and thus the composition map in the upper row equals $\hat{\Phi}^{n}_{r}h_{F}$, which finishes the proof.
\end{proof}

We are able to say more in the case of flag polytopes. Namely, the next result holds.

\begin{lemm}\label{FlagCommuteEmbed}
Suppose $P^n$ is a flag polytope and $F^r\subset P^n$ is its $r$-dimensional face. Then $\Phi^{n}_{r}(K^{r-1})=K_{n,r}$.
Moreover, there is a commutative diagram
$$\begin{CD}
  \mathcal Z_{F^r} @>\hat{i}_{r}^{n}>> \mathcal Z_{P^n}\\
  @VVH_{1} V\hspace{-0.2em} @VVH_{2} V @.\\
  \mathcal Z_{K^{r-1}} @>\hat{\Phi}_{r}^{n}>>\mathcal Z_{K^{n-1}},
\end{CD}\eqno 
$$
where $H_{1}, H_{2}$ are homeomorphisms, and $\hat{i}_{r}^{n}$ induces a split epimorphism in cohomology.
\end{lemm}
\begin{proof}
Let us prove the first part of the statement. Due to Construction~\ref{mapmcxs}, for any simple polytope $P^n$ one has the following inclusion of simplicial complexes, both on the vertex set $\phi^{n}_{r}[m(r)]$ in $K^{n-1}$:
$$
\Phi^{n}_{r}(K^{r-1})\subseteq K_{n,r}
$$
We need to prove the inverse inclusion holds. We argue by contradiction; suppose $\{j_{1},\ldots,j_k\}\subset\phi_{r}^{n}[m(r)]$ is such that:
$$
\sigma=(j_{1},\ldots,j_k)\in K_{n,r},\,\sigma\notin\Phi^{n}_{r}(K^{r-1}).
$$
These formulae are equivalent to the following relations in the face poset of $P^n$:
$$
F_{j_{1}}\cap\ldots\cap F_{j_k}\neq\varnothing,\,F^{r}\cap F_{j_1}\cap\ldots\cap F_{j_k}=\varnothing.
$$
Thus $F_{j_s}\cap F_{j_t}\neq\varnothing$ for any $s,t\in [k]$, and
$$
F_{i_1}\cap\ldots\cap F_{i_{n-r}}\cap F_{j_1}\cap\ldots\cap F_{j_k}=\varnothing.
$$
On the other hand, by definition of $\phi_{r}^{n}$ one gets the following relations in the face poset of $F^r$
$$
F_{j_s}\cap F_{i_{1}}\cap\ldots\cap F_{i_{n-r}}=G_{\alpha_s},
$$
where $\phi_{r}^{n}(\alpha_{s})=j_s$ for $s\in [k]$. This implies that $F_{j_s}\cap F_{i_t}\neq\varnothing$ for any $s\in [k], t\in [n-r]$.
Moreover, $F_{i_s}\cap F_{j_r}\neq\varnothing$ for any $r,s\in [n-r]$, since $F_{i_{1}}\cap\ldots\cap F_{i_{n-r}}=F^r\neq\varnothing$.

Now recall that $P^n$ is a flag polytope and we proved that any two of its facets from the set
$$
F_{i_1},\ldots,F_{i_{n-r}},F_{j_1},\ldots,F_{j_k}
$$
have a nonempty intersection. It follows that
$$
F_{i_1}\cap\ldots\cap F_{i_{n-r}}\cap F_{j_1}\cap\ldots\cap F_{j_k}\neq\varnothing
$$
and we get a contradiction.

To prove the second part of the statement note firstly that the diagram 
$$
\xymatrix{
& \mathcal Z_F\ar[rr]^{\hat{i}^{n}_{r}} \ar[dl]_{h_F} && \mathcal Z_P\ar[dr]^{h_P} \\
Z_{K^{r-1}} \ar[drr] &&&& \mathcal Z_{K^{n-1}}\ar[dll]_{r_{\phi_{r}^{n}[m(r)]}}\\
&& Z_{K_{n,r}}
}
$$
commutes due to Proposition~\ref{GenCase}, where $r_{\phi_{r}^{n}[m(r)]}$ is a retraction map for the induced embedding $\hat{j}_{\phi^{n}_{r}[m(r)]}$ of moment-angle-complexes. Note that one has the equality $\Phi_{r}^{n}(K^{r-1})=K_{n,r}$ that we already proved above. Thus the above diagram implies $\mathcal Z_F$ is a retract of $\mathcal Z_P$. The rest follows from the fact that a homomorphism, being equivalent to a split epimorphism, is a split epimorphism itself, see Corollary~\ref{splitepi}.
\end{proof}

\begin{coro}\label{FaceInduceSplit}
Suppose $F^r$ is an $r$-dimensional face of a simple polytope $P^n$. Then the following conditions are equivalent:
\begin{itemize}
\item[(1)] For any $\{j_{1},\ldots,j_{k}\}\subset\phi_{r}^{n}[m(r)]$, if $F_{j_{1}}\cap\ldots\cap F_{j_k}\neq\varnothing$, then $F^{r}\cap F_{j_1}\cap\ldots\cap F_{j_k}\neq\varnothing$;
\item[(2)] $\Phi_{r}^{n}(K^{r-1})=K_{n,r}$;
\item[(3)] $\hat{i}_{r}^{n}:\,\mathcal Z_{F^r}\rightarrow\mathcal Z_{P^n}$ is an embedding of a submanifold having a retraction;
\item[(4)] $(i_{r}^{n})_{*}:\,H^*(\mathcal Z_{P^n})\rightarrow H^*(\mathcal Z_{F^r})$ is a split epimorphism of rings.
\end{itemize}
\end{coro}
\begin{proof}
The equivalence of conditions (1) and (2) follows directly from Construction~\ref{mapmcxs} and the proof of Lemma~\ref{FlagCommuteEmbed}.
The implications $(2)\Rightarrow (3)\Rightarrow (4)$ follow from Corollary~\ref{splitepi}. Finally, (4) implies (2), since if $\Phi_{r}^{n}(K^{r-1})$ is not a full subcomplex, then there exists $\sigma=(j_{1},\ldots,j_k)\in K_{n,r}$, $\sigma\notin\Phi_{r}^{n}(K^{r-1})$, and $\partial\sigma\subset\Phi_{r}^{n}(K^{r-1})$, which by (4) and Theorem~\ref{BPmultigrad} gives $H^*(\partial\sigma)$ is a direct summand in $H^*(\mathcal Z_{K^{n-1}})$ (and, in particular, $\beta^{-i,2J}(K^{n-1})>0$, where $J$ is an $m$-tuple of 0's and 1's such that $J_{t}=1$ if and only if $t\in\sigma$; $|J|-i-1=k-2$), a contradiction. 
\end{proof}

Therefore, when $P^n$ is a flag polytope, the map $\Phi_{r}^{n}$ from Construction~\ref{mapmcxs} establishes a simpicial isomorphism between $K_{F^r}$ and the full subcomplex of $K_{P^n}$ on the vertex set $\phi_{r}^{n}[m(r)]$. Thus, in what follows we may identify these simplicial complexes and consider the corresponding simplicial embedding $j_{r}^{n}:\,K_{F^r}\rightarrow K_{P^n}$.

Now we want to prove that the opposite statement to that in Proposition~\ref{FlagCommuteEmbed} also holds. To do this we need the next combinatorial criterion of flagness for simplicial complexes.

\begin{lemm}\label{FlagCriterionLemma}
A simplicial complex $K$ is flag if and only if $\Link_{K}(v)$ is a full subcomplex in $K$ for any vertex $v\in K$.
\end{lemm}
\begin{proof}
If $K$ is a flag complex, let us assume the converse statement holds. Then there exists a vertex $v\in K$ and a simplex $\sigma\in K$ such that $|\sigma|\geq 2$, $\partial\sigma\subseteq\Link_{K}(v)$, and $\sigma\notin\Link_{K}(v)$. Then $v\notin\sigma$, $v\cup\sigma\notin K$, and $v\cup\sigma_{i}\in K$, for any $i$, where $\sigma_{i}$ is a facet of $\sigma$. The latter means that $v\cup\sigma\in MF(K)$ having $|v\cup\sigma|\geq 3$ elements. This contradicts our assumption that $K$ is flag.

Suppose that $K$ is not flag. Then there is a minimal nonface $\{i_{1},\ldots,i_{p}\}$ of $K$ on $p\geq 3$ elements. Then $\partial\Delta_{\{i_{1},\ldots,i_k\}}$ is a full subcomplex in $K$ on the vertex set $\{i_{1},\ldots,i_{p}\}$. Observe that $\Link_{K}(i_1)$ contains all the vertices from $\{i_{2},\ldots,i_k\}$, but not the simplex $(i_{2},\ldots,i_k)\in K$. That is, $\Link_{K}(i_1)$ is not a full subcomplex of $K$ on its vertex set, which finishes the proof.
\end{proof}

Finally, we obtain the next result.

\begin{theo}\label{FlagCriterion}
The following statements are equivalent:
\begin{itemize}
\item[(1)] $P^n$ is a flag polytope;
\item[(2)] $\Phi_{r}^{n}(K^{r-1})=K_{n,r}$ for any $F^r\subset P^n$;
\item[(3)] $\hat{i}_{r}^{n}:\,\mathcal Z_{F^r}\rightarrow\mathcal Z_{P^n}$ is an embedding of a submanifold having a retraction for any $F^r\subset P^n$;
\item[(4)] $(i_{r}^{n})_{*}:\,H^*(\mathcal Z_{P^n})\rightarrow H^*(\mathcal Z_{F^r})$ is a split epimorphism of rings for any $F^r\subset P^n$.
\end{itemize}
\end{theo}
\begin{proof}
The conditions (2), (3), and (4) are equivalent by Corollary~\ref{FaceInduceSplit}. 

Suppose (2) holds. In particular, for any facet $F^{n-1}\subset P^{n}$ one has: $\Phi_{n-1}^{n}(K^{n-2})=K_{n,n-1}$. On the other hand, $\Phi_{n-1}^{n}(K^{n-2})=\Link_{K_P}(v)$, where $v\in K_P$ corresponds to $F^{n-1}\subset P^n$, see Construction~\ref{mapmcxs}. This means $\Link_{K_P}(v)$ is a full subcomplex in $K_P$ for any $v\in K_P$, therefore, by Lemma~\ref{FlagCriterionLemma}, $K_P$ is a flag simplicial complex, which implies (1). 

Finally, (1) implies (2) by Lemma~\ref{FlagCommuteEmbed}, which finishes the proof.
\end{proof}

Obviously, a face of a flag polytope is a flag polytope itself (we can also see it from the above theorem: the equivalence of (1) and (2) above shows that $K^{r-1}=K_{F^r}$ is flag being isomorphic to a full subcomplex $K_{n,r}$ in a flag simplicial complex $K^{n-1}=K_{P^n}$).

Moreover, any flag polytope $F^r$ is a proper face of another flag polytope: $F^r$ is a facet of $P^{r+1}=F^{r}\times I$. One can ask the following question that naturally arises here:

For a given flag simple polytope $F^r$ what are the obstructions on an (indecomposable) flag simple polytope $P^n$ to have $F^r$ as its proper face?

We get a topological obstruction by means of Theorem~\ref{FlagCriterion} as follows.

\begin{coro}\label{FaceObstruction}
For a flag polytope $F^r$ to be a proper face of a (flag, indecomposable) polytope $P^n$ it is necessary that $\mathcal Z_{F^r}$ is a retract of $\mathcal Z_{P^n}$.
\end{coro}

Next example represents the situation in the nonflag case.

\begin{exam}\label{prism}
Consider a triangular prism $P^3$, $m(P^3)=5$ and denote its triangular facets by $F_1$ and $F_5$. Consider its quadrangular facet $F_2$. Then $\phi_{2}^{3}[m(2)]=\{1,3,4,5\}\subset [m(3)]=[5]$, the nerve complex of $F_2$ is a boundary of a square, and the full subcomplex of $K_{P^3}$ on the vertex set $\{1,3,4,5\}$ is a boundary of the square alongside with the diagonal $\{3,4\}$. Indeed, there is no retraction of $S^{3}\times S^{5}$ to $S^{3}\times S^3$. On the other hand, for any triangular facet, say, for $F_{1}$, one has $\Phi_{2}^{3}(K_{F_1})=K_{3,2}$ and thus the retraction and the split epimorphism in cohomology both take place.
\end{exam}

\begin{rema}
For any face embedding $i_{r}^{n}:\,F^{r}\to P^n$ there is an induced embedding of quasitoric manifolds $M^{2r}(F,\Lambda_{F})\to M^{2n}(P,\Lambda)$, which is a composition of characteristic submanifold embeddings, induced by:
$$
F^{r}=F_{i_1}\cap\ldots\cap F_{i_{n-r}}\subset F_{i_2}\cap\ldots\cap F_{i_{n-r}}\subset\ldots\subset F_{i_{n-r}}\subset P.
$$
An alternative description of the induced map of quasitoric manifolds can be given similarly to that in Construction~\ref{mapmfds}, using the definition of a quasitoric manifold as a quotient space of $T^{n}\times P^{n}$, introduced in~\cite{DJ}.\\
However, the induced embeddings of quasitoric manifolds, in general, do not have retraction maps, even in the case when $P^n$ is a flag polytope, as the example below shows.
\end{rema}

\begin{exam}
Consider a $5$-gon $P_5^2$, which is embedded into $\mathbb{R}^2$ as shown in the Figure below with facets $F_1,\ldots,F_5$ labeled simply by $\{1,2,3,4,5\}$ respectively.


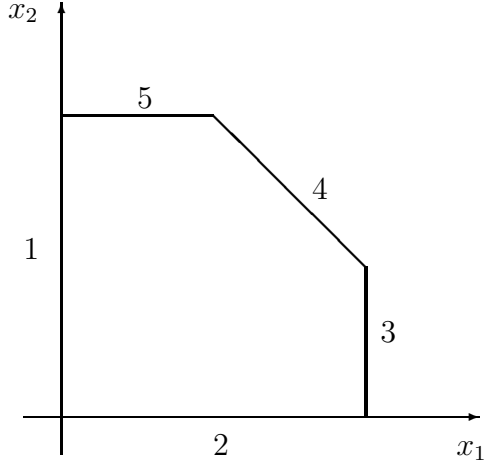
\begin{center}
\begin{picture}(90,60)(0,0)
{\thicklines
  \put(25,40){\line(1,0){20}}
  \put(65,0){\line(0,1){20}}
  \put(65.1,0){\line(0,1){20}}
  \put(65,20){\line(-1,1){20}}
  }
   \put(20,0){\vector(1,0){60}}
   \put(25,-5){\vector(0,1){60}}
  \put(77,-5){$x_1$}
  \put(18,53){$x_2$}
   \put(20,21){$1$}
   \put(45,-5){$2$}
   \put(67,10){$3$}
   \put(58,29){$4$}
   \put(35,41){$5$}
\end{picture}
\end{center}
\vskip 1.0cm

Let us take $P^2=P_5^2 = \{ x\in \mathbb{R}^2\;: Ax+b\geqslant 0 \}$, where
\[ A^\top =
  \begin{pmatrix}
    1 & 0 & -1 & -1 & 0 \\
    0 & 1 & 0 & -1 & -1 \\
  \end{pmatrix},\qquad
 b^\top = (0,0,2,3,2).
\]
and $\top$ means transposition of a matrix.
By Davis-Januszkiewicz theorem, see~\cite{DJ}, one has the following description of the cohomology ring of $M=M_{P_5}$
$$
H^*(M(P_{5}^{2},A^{\top}))\cong\mathbb{Z}[v_{1},v_{2},v_{3},v_{4},v_{5}]/I,
$$
where
$$
I=(v_{1}-v_{3}-v_{4},v_{2}-v_{4}-v_{5},v_{1}v_{3},v_{1}v_{4},v_{2}v_{4},v_{2}v_{5},v_{3}v_{5}),
$$
which implies that $v_{1}^{2}=v_{2}^{2}=0$ and $v_{3}^{2}=v_{4}^{2}=v_{5}^{2}\neq 0$ in $H^*(M)$ (if all the squares of the 2-dimensional generators were zero, then all the monomials of degree greater than 2 in $H^*(M)$ are zero, which contradicts $H^{4}(M^4)\cong\mathbb{Z}$). Therefore, there is no retraction map for the embedding of quasitoric manifolds induced by the face embedding $i_{1}^{3}:\,F_{3}\to P_5$, since the converse would imply the existence of the induced split ring epimorphism in cohomology: $H^{2}(M^4)\to H^{2}(M_{F_3})$, which contradicts $v_{3}^{2}\neq 0$ in $H^*(M)$ (recall that $M_{F_3}\cong\mathbb{C}P^1$ and thus for its 2-dimensional cohomology generator one has: $v_{3}^{2}=0$).
\end{exam}


\section{Applications I}

Now we recall a definition of a higher Massey product in cohomology of a differential graded algebra, see the exposition of basic definitions in the work of Babenko and Taimanov~\cite{BaTa}; more details can be found in~\cite[Appendix $\Gamma$]{BP04}.

\begin{defi}\label{DefiningSystem}
Suppose $(A,d)$ is a differential graded algebra, $\alpha_{i}=[a_{i}]\in H^{*}[A,d]$ and $a_{i}\in A^{n_{i}}$ for $1\leq i\leq k$.
Then a \emph{defining system} for $(\alpha_{1},\ldots,\alpha_{k})$ is a $(k+1)\times (k+1)$-matrix $C$ such that the following conditions hold:
\begin{itemize}
\item[{(1)}] $c_{i,j}=0$, if $i\geq j$,
\item[{(2)}] $c_{i,i+1}=a_{i}$,
\item[{(3)}] $a\cdot E_{1,k+1}=dC-\bar{C}\cdot C$ for some $a=a(C)\in A$, where $\bar{c}_{i,j}=(-1)^{deg{c_{i,j}}}\cdot c_{i,j}$ and $E_{1,k+1}$ is a $(k+1)\times (k+1)$-matrix with all elements equal to zero, except for that in the position $(1,k+1)$, which equals 1.
\end{itemize}

It is easy to see that conditions (1)-(3) above imply $d(a)=0$ and $a\in A^{m}$, $m=n_{1}+\ldots+n_{k}-k+2$.

A $k$-fold \emph{Massey product} $\langle\alpha_{1},\ldots,\alpha_{k}\rangle$ is said to be \emph{defined}, if there exists a defining system $C$ for it.\\
If so, this Massey product is defined to be the set of all cohomology classes $\alpha=[a(C)]$, when $C$ is a defining system. A defined Massey product is called \emph{trivial}, or \emph{vanishing} if $[a(C)]=0$ for some defining system $C$.
\end{defi}

We next recall the construction of a sequence of 2-truncated cubes $\mathcal Q=\{Q^n|\,n\geq 0\}$, for which $\mathcal Z_{Q^{n}}$ was proved in~\cite{L2} to have a nontrivial $n$-fold Massey product in cohomology for all $n\geq 2$.

\begin{defi}[\cite{L2}]\label{2truncMassey}
Set $Q^0$ to be a point and $Q^1\subset\R^1$ to be a segment $[0,1]$. Suppose $I^{n}=[0,1]^n, n\geq 2$ is an $n$-dimensional cube with facets $F_{1},\ldots,F_{2n}$ such that $F_{i},1\leq i\leq n$ contains the origin 0, $F_{i}$ and $F_{n+i}$ for $1\leq i\leq n$ are parallel.
Then its face ring is the following one:
$$
\ko[I^n]=\ko[v_{1},\ldots,v_{n},v_{n+1},\ldots,v_{2n}]/I_{I^n},
$$
where $I_{I^n}=(v_{1}v_{n+1},\ldots,v_{n}v_{2n})$.

Consider the polynomial ring
$$
\mathbb{Z}[v_{1},\ldots,v_{2n},w_{k',n+k'+i'}|\,1\leq i'\leq n-2, 1\leq k'\leq n-i']
$$
and the following square free monomial ideal
$$
I=(v_{k}v_{n+k+i},w_{k',n+k'+i'}v_{n+k'+l},w_{k',n+k'+i'}v_{p},w_{k',n+k'+i'}w_{k'',n+k''+i''}),
$$
in the above ring, where $v_{j}$ corresponds to $F_{j}$ for $1\leq j\leq 2n$, and
$$
0\leq i\leq n-2, 1\leq k\leq n-i, 1\leq i',i''\leq n-2, 1\leq k'\leq n-i',
$$
$$
1\leq k''\leq n-i'', 1\leq p\neq k'\leq k'+i', 0\leq l\neq i'\leq n-2,
$$
$$
k'+i'=k''\,\text{or }k''+i''=k'.
$$

Let us define $Q^n\subset\R^n$ to be a simple polytope such that $I_{Q^n}=I$. Observe that $Q^n$ has a natural realization as a 2-truncated cube and, furthermore, its combinatorial type does not depend on the order of face truncations of $I^n$.
\end{defi}

Below we give an explicit description for $Q^n\subset\R^n$ and the face embeddings $i_{n-1}^{n}:\,Q^{n-1}\to Q^n$ in the ambient Euclidean spaces, see Example~\ref{faceQ}.

The next result on higher Massey products in cohomology of moment-angle manifolds holds.

\begin{theo}[{\cite{L2}}]\label{mainMassey}
Let $\alpha_i\in H^{3}(\mathcal Z_{Q^n})$ be represented by a 3-cocycle $v_{i}u_{n+i}\in R^{-1,4}(Q^n)$ for $1\leq i\leq n$ and $n\geq 2$. Then all Massey products of consecutive elements from $\alpha_{1},\ldots,\alpha_{n}$ are defined and the whole $n$-product $\langle\alpha_{1},\ldots,\alpha_{n}\rangle$ is nontrivial.
\end{theo}

\begin{theo}\label{Qdfpnm}
For any $Q^n\in\mathcal Q$ and any $0\leq r<n$ there exists a face $F^{r}\subset Q^n$ such that $F^r$ is combinatorially equivalent to $Q^r$.
\end{theo}
\begin{proof}
To prove the statement it suffices to show that $Q^{n-1}$ is a facet of $Q^n$ for $n\geq 1$.

Indeed, consider the facet $F_{n-1}$ of $Q^n$. Let us show that $F_{n-1}=Q^{n-1}$. By definition of $Q^n$, $F_{n-1}$ is obtained from a facet $G_{n-1}\simeq I^{n-1}$ of $I^{n}$ with facets $G_{i},1\leq i\leq 2n$ (recall that $G_{i}\cap G_{n+i}=\varnothing$ for $1\leq i\leq n$ as in Definition~\ref{2truncMassey}) by consecutive cutting off faces of codimension 2 and, moreover, is also a 2-truncated cube, see~\cite[\S1.6]{TT}. The latter faces are the following ones (here we use the fact that cutting off a facet does not change combinatorial type of a polytope and that $G_{n-1}\cap G_{2n-1}=\varnothing$):
$$
G_{n-1}\cap G_{i}\cap G_{j},
$$
where $1\leq i\leq n-2$, $n+2\leq j\leq 2n$, $j\neq 2n-1$. It remains to observe that the Stanley-Reisner ideals $I_{F_{n-1}}$ and $I_{Q^{n-1}}$ are isomorphic by Definition~\ref{2truncMassey} and $m(Q^{n-1})=m(F_{n-1})$. Therefore, $\mathbb{Z}[F_{n-1}]\cong\mathbb{Z}[Q^{n-1}]$, which implies that the resulting polytope $F_{n-1}$ is combinatorially equivalent to $Q^{n-1}$. Similarly, $F_{2n-1}\simeq Q^{n-1}$.

Moreover, one can easily see from the above argument, that $F_{r}\cap F_{r+1}\cap\ldots\cap F_{n-1}\simeq Q^{r}$.
\end{proof}

Now Theorem~\ref{FlagCriterion} and the above theorem imply the following statement holds.

\begin{coro}\label{AllProducts}
There exists a nontrivial $k$-fold Massey product in $H^*(\mathcal Z_{Q^n})$ for each $k, 2\leq k\leq n$.
\end{coro}

\begin{rema}
Note that it was proved in~\cite[Theorem 4.1]{L3} that the full subcomplex in $K=K_{Q^n}$ on the vertex set $\{1,\ldots,r-1,n,n+1,\ldots,n+r-1,2n\}$ is combinatorially equivalent to the full subcomplex of $K_{Q^r}$ on the vertex set $[2r]$ for each $r, 2\leq r\leq n$, which provides a different proof of the above statement.
\end{rema}

\begin{exam}\label{faceQ}
Let us give explicit formulae for the embedding $\mathcal Z_{Q^r}\rightarrow\mathcal Z_{Q^n}$, where $Q^{r}\simeq F^{r}=F_{r}\cap\ldots\cap F_{n-1}$, using Construction~\ref{mapmfds2}.
Note that we can choose the basis in $\R^n$ such that $A_{Q^n}^{T}=(E_{n},-E_{n},B^T)$, where $B$ is an $(m(n)-2n)\times n$-matrix such that arrows of $B$ have the form $e_{k}-e_{k+i},1\leq i\leq n-2,1\leq k\leq n-i$; they correspond to the codimension 2 face truncations, which are performed on $I^n$ to obtain $Q^n$, see Definition~\ref{2truncMassey}.

It is easy to see from the proof of Theorem~\ref{Qdfpnm} that a row of $\tilde{A}_{I}$ has either a form $e_{k}-e_{k+i}$ with $1\leq k\leq n-2$ and $2\leq k+i\leq n$, $k+i\neq n-1$, or $e_k, k\in\{1,\ldots,r-1,n\}$ and $-e_{k}$ for $k\in\{n+1,\ldots,n+r-1,2n\}$. Moreover, the map $i^{n}_{r};\,\R^r\to\R^n$ sends $(x_{1},\ldots,x_{r})$ to $(x_{1},\ldots,x_{r},0_{n-r})$.

Suppose $n=4,r=3$. Then we can give $W\subset\C^{13}$ by the following formulae:
$$
(|z_{1}|^{2},|z_{4}|^{2})=(x_{1},x_{4}), (|z_{5}|^{2},|z_{8}|^{2})=(-x_{1}+1,-x_{4}+1),
$$
$$
(z_{2},z_{3},z_{6},z_{7})=(0,0,0,0),
$$
$$
|z_{9}|^{2}=x_{1}-x_{2}+1,|z_{10}|^{2}=x_{2}-x_{4}+1,|z_{11}|^{2}=x_{2}-x_{3}+1,
$$
$$
|z_{12}|^{2}=x_{3}-x_{4}+1,|z_{13}|^{2}=x_{1}-x_{3}+1.
$$
\end{exam}


\section{Applications II}

Starting with any indecomposable flag polytope $F^r$ we can construct a sequence of indecomposable flag polytopes $\{P^n|\,n\geq r\}$ with $P^{r}=F^{r}$ such that $P^{k}$ is a face of $P^{l}$ for any $l>k\geq r$.

\begin{constr}\label{familyFlag}
Given a flag polytope $F^r$ let us determine a sequence of flag polytopes $\mathcal P(F)=\{P^n|\,n\geq r\}$ as follows. Set $P^{r}=F^{r}$ and if $P^n$ is already constructed, define $P^{n+1}$ to be obtained from $P^{n}\times I$ by cutting of a certain codimension-2 face $F_{i(n)}\times\{1\}\subset P^{n}\times\{1\}\subset P^{n}\times I$ ($F_{i(n)}$ is a facet of $P^n$). Observe, that the resulting polytope is again flag and has $P^n$ as its facet $P^{n}\times\{0\}\subset P^{n+1}$ for any $n\geq r$. Obviously, the combinatorial type of $P^n$ for $n>r$ depends, in general, not only on that of $F^r$, but also on the choice of the facet $F_{i(n)}$ of $P^{n}$.
\end{constr}

The above construction introduces a new operation $\fc$ ('face cut') on flag simple polytopes; we have $P^{n+1}=\fc(P^n)$ for all $n\geq r$. We denote by $Q=\fc^{k}(P)$ a polyope obtained from $P$ by performing $k$ consecutive operations described in Construction~\ref{familyFlag}; note that $P=\fc^{0}(P)$ and $\dim Q=\dim P+k$. 

As an application of the above construction one immediately obtains the next result.

\begin{coro}\label{MasseySeq}
Suppose $F^r$ is a flag polytope and there exists a nontrivial $k$-fold Massey product in $H^*(\mathcal Z_{F^r})$. Then there is a sequence of polytopes $\mathcal P=\{P^n|\,n\geq r\}$ such that there exists a nontrivial $k$-fold Massey product in $H^*(\mathcal Z_{P^n})$ for all $n\geq r$. 
\end{coro}
\begin{proof}
Direct application of Theorem~\ref{FlagCriterion} to the sequence of flag polytopes $\mathcal P(F^r)=\{P^n|\,n\geq r\}$ defined in Construction~\ref{familyFlag}.
\end{proof}

\begin{defi}
A sequence of indecomposable flag polytopes $\mathcal P=\{P^n\}$ such that there exists a nontrivial $k$-fold Massey product in $H^*(\mathcal Z_{P^n})$ with $k\to\infty$ as $n\to\infty$ and, moreover, the existence of a nontrivial $k$-fold Massey product in $H^*(\mathcal Z_{P^n})$ implies existence of a nontrivial $k$-fold Massey product in $H^*(\mathcal Z_{P^l})$ for any $l>n$ will be called \emph{a sequence of polytopes with strongly connected Massey products}.
\end{defi}

\begin{defi}
We say that sequences of polytopes $\mathcal P_{1}=\{P_{1}^{n}\}$ and $\mathcal P_{2}=\{P_{2}^{n}\}$ are \emph{combinatorially different} if for any $N\geq 0$ there exists $n>N$ such that $P_{1}^{n}$ and $P_{2}^{n}$ are not combinatorially equivalent.
\end{defi}

\begin{prop}\label{MasseySeqInfinity}
There exists an infinite set $\mathcal S=\{\mathcal P_{\alpha}|\,\alpha\in I\}$ of sequences of polytopes with strongly connected Massey products such that any $P_{\alpha}\in\mathcal S$ is combinatorially different from all other elements of $\mathcal S$.
\end{prop}
\begin{proof}
Observe that if $P^{k}, P^{l}\in\mathcal P(F)$ ($k<l$) for some flag polytope $F$, then $P^{k}$ is a proper face of $P^l$ and, furthermore, $\fc(P^k)=P^{k+1}$ is a face of $P^l\subset\fc(P^l)$. For any sequence $s\in\{0,1\}^{\infty}$ not stabilizing at 0 define such a sequence of flag polytopes $\mathcal P_{s}=\{P_{s}^{n}|\,n\geq 4\}$ that $P_{s}^{n}=Q^{n}$, if $s_{n}=1$, and $P_{s}^{n}=\fc^{n-k(n)}(Q^{k(n)})$, if $s_{n}=0$. Here $k(n)=1$, if $s_{1}=\ldots=s_{n}=0$, and $k(n)=\max\,\{m|\,m<n, s_{m}=1\}$, otherwise. Then Theorem~\ref{FlagCriterion} implies $\mathcal P_{s}$ is a sequence of polytopes with strongly connected Massey products. Moreover, a sequence $\mathcal P_{s}$ is combinatorially different from $\mathcal P_{s'}$ (for $s\neq s'$) if $|s-s'|$ is not stabilizing at 0, since $m(\fc(P))=m(P)+3$ for any polytope $P$ and $m(Q^n)=\frac{n(n+3)}{2}-1$. 
\end{proof}

Another way to get different sequences of flag polytopes with higher nontrivial Massey products in cohomology of their moment-angle manifolds is introduced in the next statement.
 
\begin{theo}\label{mainMasseysequence}
There exists infinitely many sequences $\mathcal P_{k}=\{P_{k}^n\}, k\geq 1$ of indecomposable flag simple polytopes such that 
\begin{itemize}
\item[(a)] If $P\in\mathcal P_{i}$ and $Q\in\mathcal P_{j}$ for $i\neq j$, then $P$ and $Q$ are not combinatorially equivalent;
\item[(b)] For any $k\geq 1$ and $r\geq 2$ there exists $N=N(k,r)$ such that there is a nontrivial $l$-fold Massey product in $H^*(\mathcal Z_{P_{k}^n})$, for all $2\leq l\leq r, n\geq N$.
\end{itemize}
\end{theo}
\begin{proof}
Consider the following sequences of flag polytopes: $\mathcal P_{k}=\{\fc^{k-1}(Q^{n})|\,n\geq 3\}$, $k\geq 1$. 

To prove (a) assume the converse is true; then the following equality for the number of facets of combinatorially equivalent polytopes of the same dimension holds:
$$
m(\fc^{k}(Q^l))=m(\fc^{l}(Q^k)).
$$
As $m(Q^n)=\frac{n(n+3)}{2}-1$ for any $n\geq 2$ and $m(\fc(P))=m(P)+3$ the above formula implies $(l-k)(l+k-3)=0$. If $l=k$, then both polytopes belong to the same sequence $\mathcal P_{l+1}$ and if $l+k=3$, then one of the dimensions $l$ or $k$ is smaller than 3. In both cases we get a contradiction with the definition of the sequences $\mathcal P_k$, which finishes the proof of (a).

Applying Construction~\ref{familyFlag}, one has that $Q^n$ is a face of any polytope of dimension greater or equal to $n+k-1$ in $\mathcal P_k$. Therefore, statement (b) (with $N(k,r)=r+k-1$) follows from Theorem~\ref{FlagCriterion} and Corollary~\ref{AllProducts}, which finishes the whole proof.
\end{proof}


\end{document}